\documentclass{amsart}
\usepackage{epsfig}
\usepackage{amsmath}
\usepackage{amssymb}
\usepackage{amscd}
\usepackage{graphicx}
\usepackage{xcolor}
\usepackage{float}
\usepackage{verbatim}
\usepackage{bbm}
\usepackage{cancel}
\usepackage{enumerate}
\usepackage{hyperref}

\newtheorem{thm}{Theorem}[section]
\newtheorem{lem}[thm]{Lemma}

\newtheorem{defn}[thm]{Definition}

\theoremstyle{remark}

\def \N {\mathbb N}

\def \Z {\mathbb Z}
\def \R {\mathcal R}

\def \A {\mathcal A}
\def \B {\mathcal B}
\def \C {\mathcal C}
\def \G {\mathcal G}
\def \M {\mathcal M}

\def \U {\mathcal U}
\def \P {\mathcal P}

\def \eps {\varepsilon}

\def \sq {sequence}

\def \tl {topological}
\def \im {invariant measure}
\def \inv {invariant}

\def \dist{\mathsf{dist}}
\def \emp{\mathsf{emp}}

\author{Tomasz\ Downarowicz, Jean-Paul Thouvenot and Benjamin Weiss}
\address{Faculty of Pure and Applied Mathematics, Wroc\l aw University of Science and Technology, Wroc\l aw, Poland}\email{downar@pwr.edu.pl}
\address{Laboratoire de Probabilit\'es, Universit\'e Statistique et Mod\'elisation, Sorbonne Universit\'e, 4 Place Jussieu, 75252 Paris Cedex 05, France}\email{jean-paul.thouvenot@upmc.fr}
\address{Mathematics Institute, Hebrew University of Jerusalem, Jerusalem, Israel}\email{weiss@math.huji.ac.il}

\title{On a theorem of Dan Rudolph: Part II: Amenable groups}

\begin{document}

\begin{abstract}We prove an analog of Rudolph's theorem for actions of countable amenable groups, which asserts that among \im s with entropy at least $c$ on the $G$-shift $(\Lambda^G,\sigma)$, a typical measure has entropy $c$ and is Bernoulli. We also address a relative version of this theorem.  
\end{abstract}

\thanks{The research of the first author was supported
by the NCN (Polish National Science Center) Grant No. 2022/47/B/ST1/02866.}

\subjclass[2020]{Primary 37A05, 37A35, 43A07; Secondary 37A25, 37A15}
\keywords{Amenable group action, entropy, generic properties, Bernoulli systems, relative Bernoulli, very weak Bernoulli, orbit equivalence}

\maketitle

\section{Introduction}
In part I of this paper the second and third authors gave a proof of the following theorem that Dan Rudolph presented in the Dynamics Seminar at the University of Maryland 
on Seprtember 11, 2003:

\begin{thm} For a finite set $A$ let $X = A^\Z$ and denote by $S$ the shift.  For $0<c\le\log|A|$  denote by $\M_c$ the set of $S$-invariant measures with entropy greater than or equal to $c$. There is a dense $G_\delta$ subset $C\subset  \M_c$ such that for all $\mu\in C$ the system $(X,\mu,S)$ has entropy equal to $c$ and is isomorphic to a Bernoulli shift.
\end{thm}

The aim of this second part is to prove an analogous result, where integers are replaced by an arbitrary countable amenable  group $G$. Often, extensions 
to amenable groups turn out to be routine. However, in this case there is a real difficulty. 
For instance, the simple (for $\Z$-shifts) task of partitioning a symbolic \sq\ into long blocks, in case of a $G$-shift becomes much more intricate and requires the machinery of tilings.
For $\Z$-actions, in order to show that Bernoulli measures form a $G_\delta$-set, we used the characterization of Bernoulli processes as those processes that are very weakly Bernoulli (VWB). This notion depends in an essential way on the fact that for processes indexed by the integers there is a notion of the ``past'' of a process, and the process entropy can be defined as the conditional entropy of the present given the past. For a general amenable group there is no such corresponding notion. 
  
In order to circumvent this problem, we use the  method introduced in \cite{9a} and the characterization of Bernoulli extensions as those satisfying the 
relatively VWB property introduced in \cite{11a}. In essence, this amounts to proving a relative version of Rudolph's theorem for the integers and then utilizing the orbit equivalence between all ergodic amenable group actions.      

The precise formulation of the main theorem of this paper, the generalization of Rudolph's theorem to actions of countable amenable groups, is as follows:

\begin{thm}\label{Ru}
Let $\Omega=\Lambda^G$, where $\Lambda$ is a finite alphabet and $G$ an infinite countable amenable group, and let $\sigma$ denote the $G$-shift on $\Omega$. For a positive number $c\le\log|\Lambda|$, let $\M_c$ denote the set of all $\sigma$-\im s on $\Omega$ with the Kolmogorov-Sinai entropy at least $c$. Then Bernoulli measures $\mu$ (i.e., such that the measure-preserving $G$-system $(\Omega,\mu,\sigma)$ is isomorphic to an i.i.d.\ $G$-process) of entropy equal to $c$ is a dense $G_\delta$ subset of $\M_c$.
\end{thm}
Theorem~\ref{Ru} is a direct consequence of Theorems~\ref{T1} and~\ref{T2} proved in Sections~\ref{S3} and~\ref{S4}, and asserting, respectively, that Bernoulli measures with entropy $c$ are dense in $\M_c$, and that Bernoulli measures with entropy $c$ form a $G_\delta$-subset of $\M_c$. 
\medskip

The tools developed to prove Theorem~\ref{Ru} allow one to prove, with almost no additional effort, a relative version of this theorem.
This is done on Section~\ref{S5}, to be specific, we will sketch the proof of the following:
\begin{thm}\label{Ru1}
Let $\Lambda$ and $\Lambda_1$ be finite alphabets and let $\Omega=\Lambda^G$ and $\Omega_1=\Lambda_1^G$. On both spaces we consider the action of the $G$-shift $\sigma$. Let $\pi:\Omega\to\Omega_1$ be a \tl\ factor map. Let $\mu_1$ be an ergodic $\sigma$-\im\ on $\Omega_1$. For a positive number $c\le\sup\{h(\mu|\Omega_1):\mu\in\M, \pi^*(\mu)=\mu_1\}$, let $\M^{\mu_1}_c$ denote the set of all $\sigma$-\im s $\mu$ on $\Omega$ such that $\pi^*(\mu)=\mu_1$ and such that the conditional entropy $h(\mu|\Omega_1)$ is greater than or equal to $c$. Then measures $\mu$ which are relatively Bernoulli with respect to $\mu_1$ and such that $h(\mu|\Omega_1)=c$ form a dense $G_\delta$ subset~of~$\M^{\mu_1}_c$.
\end{thm}

\section{Preliminaries}\label{S2}
Let $G$ be a countable group with unity $e$. Given a finite set $K\subset G$ and $\gamma>0$, we will say that a finite set $F\subset G$ is \emph{$(K,\gamma)$-\inv} if 
$$
\frac{|KF\triangle F|}{|F|}<\gamma
$$
($\cdot\triangle\cdot$ denotes the symmetric difference of sets and $|\cdot|$ stands for cardinality of a set).
A \emph{F\o lner \sq} in $G$ is any \sq\ $\{F_n\}_{n\ge1}$ of finte subsets of $G$ such that for every finite $K\subset G$ and any $\gamma>0$ the sets $F_n$ are eventually (i.e., for sufficiently large $n$) $(K,\gamma)$-\inv. Amenability of $G$ is equivalent to the existence of a F\o lner \sq. From now on we assume that $G$ is amenable and we fix some F\o lner \sq\ $\{F_n\}_{n\ge1}$ in $G$.
\smallskip

Let $(X,\mu)$ be a standard probability space\footnote{By convention, we skip marking the sigma-algebra, but we will refer to its elements as \emph{measurable sets}.}.
\begin{defn}\label{d1}
Let $\P=\{P_1,P_2,\dots,P_n\},\ \P'=\{P'_1,P'_2,\dots,P'_n\}$ be two measurable partitions of $(X,\mu)$ of the same cardinality $n$. We let
$$
d_1(\P,\P')=\inf_\pi\sum_{i=1}^n \mu(P_i\triangle P'_{\pi(i)}),
$$
where $\pi$ ranges over all permutations of $\{1,2,\dots,n\}$.
\end{defn}
 
Let $\tau:G\times X\to X$  be a measure-preserving action. When the system $(X,\mu,\tau)$ is fixed, we will write $gx$ in place of $\tau(g,x)$ ($g\in G,\ x\in X$) and $FA$ instead of $\{gx:g\in F, x\in A\}$, ($F\subset G$, $A\subset X$). The action is said to be \emph{free} if every $g\in G\setminus\{e\}$ satisfies $\mu\{x:gx=x\}=0$.
\smallskip

Let $\Lambda$ be a finite set called \emph{alphabet}. On $\Omega=\Lambda^G$ we will consider the action of the $G$-shift $\sigma:G\times\Omega\to\Omega$ given by
$$
(g\omega)_h=\omega_{hg}, \text{ where } \omega=(\omega_h)_{h\in G}\in\Omega,\text{ and } g\in G
$$
(by the convention introduce earlier, $g\omega$ stands for $\sigma(g,\omega)$). 
\smallskip

For $\omega\in\Omega$ and a set $K\subset G$, by $\omega|_K$ we will denote the restriction of $\omega$ to $K$ (note that $\omega$ is a function from $G$ to $\Lambda$). If $K=\{g\}$ is a singleton ($g\in G$), we will write $\omega_g$ rather than $\omega|_{\{g\}}$.

If $K$ is finite, $\omega|_K$ is called a \emph{block}. In general, a block over $K$ is an element $\alpha\in\Lambda^K$, i.e., a function $\alpha:K\to\Lambda$. We will often identify blocks up to shifting, that is we will treat $\Lambda^K$ and $\Lambda^{Kg}$ ($g\in G$) as identical sets.
We will say that a block $\alpha\in\Lambda^K$ \emph{appears} in $\omega$ if $\omega|_{Kg}=\alpha$ for some $g$. If $F\subset G$ is another finite set and $\beta\in\Lambda^F$ is a block, the \emph{empircal measure} on $\Lambda^K$ associated with $\beta$ is defined by frequencies:
$$
\emp_\beta(\alpha)=\frac1{|F|}|\{g\in F: Kg\subset F, \beta|_{Kg}=\alpha\}|, \ \ \alpha\in\Lambda^K.
$$
We can also associate empirical measures with elements $\omega\in\Omega$. For that, we need to refer to our fixed F\o lner \sq\ $\{F_n\}_{n\ge1}$ in $G$ and let
$$
\emp_\omega(\alpha)=\lim_{n\to\infty}\emp_{\omega|_{F_n}}(\alpha),
$$
provided that the limit exists. 
\begin{defn}\label{gen}
A point $\omega$ is \emph{generic} for a $\sigma$-\im\ $\mu$ on $\Omega$ if the empirical measures $\emp_\omega$ on $\Lambda^K$ exist for all finite sets $K\subset G$, and for every $K$ and $\alpha\in\Lambda^K$ it holds that
$$
\emp_\omega(\alpha) = \mu([\alpha]),
$$ 
where $[\alpha]$ denotes the \emph{cylinder} pertaining to the block $\alpha$, $[\alpha]=\{\omega\in\Omega:\omega|_K=\alpha\}$. 
\end{defn}
By the Lindenstrauss' pointwise ergodic theorem for amenable groups \cite[Theorem~1.3]{6a}, if we choose our F\o lner \sq\ $\{F_n\}_{n\in\N}$ to be tempered (which we always can), a generic point exists for every ergodic measure, in particular, for any Bernoulli measure.
\smallskip

In the sequel, by a slight abuse of notation, we will write $\mu(\alpha)$ rather than $\mu([\alpha])$. Observe that by the same abuse of notation, for every finite set $K\subset G$, $\Lambda^K$ is a clopen partition of $\Omega$ (i.e., with clopen atoms). 

\subsection{Ananlog of the Rokhlin Lemma} We will use several times \cite[Theorem~3.6]{2a}, which we rewrite adapted to the notation used throughout the paper. 
\begin{thm}\label{conley}
Consider a free action $\tau$ of $G$ on a standard probability space $(X,\mu)$.
For any finite set $K$ and any $\gamma>0$, there exists a collection $\{B_i: 1\le i\le m\}$ of measurable subsets of $X$, and a 
collection $\{S_i: 1\le i\le m\}$ of $(K,\gamma)$-\inv\ subsets of $G$, each containing $e$, and such that $\{gB_i:1\le i\le m,\ g\in S_i\}$ partitions $X$ (up to measure zero).
\end{thm}

The above theorem is a generalization of Rokhlin's lemma known for actions of $\Z$ (in which case $m=2$ is sufficient). The sets $S_iB_i$ are interpreted as \emph{towers} (we have $m$ disjoint towers), and, for given $1\le i\le m$, the sets $gB_i$ with $g\in S_i$ are called \emph{levels} of the $i$th tower (there are $|S_i|$ disjoint levels). The set $B_i$ (which is a level, because $e\in S_i$) is called the \emph{base} of the $i$th tower. 
In this context, the usual notion of height of a tower is insufficient, and we will refer to the set $S_i$ as the \emph{shape} of the $i$th tower.

\subsection{Tiling factor} In addition to the sets $B_i$ with $1\le i\le m$, let us also define $B_0=X\setminus\bigcup_{i=1}^mB_i$. One can create a natural symbolic factor of the system $(X,\mu,\tau)$, over the alphabet $\{0,1,2,\dots,m\}$, via the factor map $\psi:X\to\bar\Omega$, where $\bar\Omega=\{0,1,2,\dots,m\}^G$, defined by the rule
$$
\psi(x)_g = i \iff gx\in B_i, \ \ 0\le i\le m.
$$
In other words, $\psi(x)$ is the \emph{$\P$-itinerary} of $x$, where $\P$ is the partition of $X$ by the sets $B_i$, $i\in\{0,1,\dots,m\}$. The push-down of $\mu$,
$\psi^*(\mu)$, is a shift-\im\ on $\bar\Omega$ and will be denoted by $\nu$. The partition condition of Theorem~\ref{conley} implies what we call the \emph{tiling property} of the factor $(\bar\Omega,\nu,\sigma)$ with respect to the shapes $S_1,S_2,\dots,S_m$:
\begin{multline}\label{tp}
\text{For $\nu$-almost every $\bar\omega\in\bar\Omega$, the family $\{S_ig: 1\le i\le m, \bar\omega_g=i\}$}\\ \text{is a partition of $G$.}
\end{multline}
This partition divides $G$ into ``tiles'', each being a shifted copy of one of the finitely many shapes~$S_i$. The symbolic measure-preserving system $(\bar\Omega,\nu,\sigma)$ will be called a \emph{tiling factor} of $(X,\mu,\tau)$. The same symbolic system satisfying~\eqref{tp}, considered in isolation will be called a \emph{tiling system} and its elements will be denoted by $\bar\omega=(\bar\omega_g)_{g\in G}$.

Conversely, any factor map $\psi:X\to\bar\Omega$ of $(X,\mu,\tau)$ to a tiling system $(\bar\Omega,\nu,\sigma)$ with shapes $S_1,S_2,\dots,S_m$, determines the towers specified in Theorem\ref{conley}. Namely, the bases of the towers are the sets $B_i=\psi^{-1}([i])=\{x:\psi(x)_e=i\}$ ($1\le i\le m$), and the shapes of the towers are the respective sets $S_i$. The partition condition of Theorem~\ref{conley} follows immediately from the tiling property \eqref{tp}.

\smallskip
When, instead of a general system $(X,\mu,\tau)$, we consider a symbolic $G$-shift $(\Omega,\mu,\sigma)$ where $\Omega=\Lambda^G$ and $\mu$ is free, and we pick an $\omega\in\Omega$ in the domain of the tiling factor $\psi:\Omega\to\bar\Omega$, then the blocks $\omega|_{S_ig}$, where $g$ is such that $\psi(\omega)_g=i$ ($1\le i\le m$), are called \emph{tiles} of $\omega$. The set $S_i$ is called the \emph{shape} of the tile, and $g$ is its \emph{center}. Almost every element $\omega\in\Omega$ is a \emph{concatenation}\footnote{Formally, a concatenation of blocks $\alpha_\kappa$ over disjoint domains $K_\kappa\subset G$ ($\kappa$ is ranging over a finite or countable set) is the function defined on the union $\bigcup_\kappa K_\kappa$ whose graph is the union of the graphs of the blocks $\alpha_\kappa$.} of its tiles.
\smallskip

Let $\M$ denote the collection of all $\sigma$-\im s on $\Omega$.
In what follows, we will use a convex metric on $\M$ compatible with the weak* topology. Given such a metric and an $\eps>0$, there exists a finite set $K$ and $\delta>0$ such that if two measures agree up to $\delta$ on every block $\alpha\in\Lambda^K$  then these measures are $\eps$-close. The blocks in $\Lambda^K$ will be referred to as \emph{test blocks}. Also, if $F$ is a large finite set, we will say that a block $\beta\in\Lambda^F$ is $\eps$-close to a measure $\mu$ if its empirical measure $\emp_\beta$ on test blocks agrees  up to $\delta$ with the values assigned by $\mu$. 


\section{Bernoulli measures are dense in $\M_c$}\label{S3}

We assume familiarity of the reader with the notions of measure-theoretic entropy and conditional entropy. We will use the following notation.
\begin{itemize}
\item If the action $\tau$ on a probability space $(X,\mu)$ is fixed then the Kolmogorov--Sinai entropy of the system $(X,\mu,\tau)$ will be denotet shortly by $h(\mu)$. This notation is especially convenient in the context of \tl\ systems supporting numerous \im s. Then we can talk of the \emph{entropy function} $h$ on \im s. It is well known that for the shift system $(\Omega,\sigma)$ the entropy function is upper semicontinuous on $\M$. 
\item If $\P$ is a finite measurable partition of $X$ then the Shannon entropy of $\P$ will be denoted by $H_\mu(\P)$ while the entropy of the process generated by $\P$ will be denoted by $h_\mu(\P)$.
\item Finally, if $\pi:X\mapsto X'$ is a factor map between two measure-preserving $G$-systems $(X,\mu,\tau)$ and $(X',\mu',\tau')$ (note that then $\mu'=\pi^*(\mu)$), then $h(\mu|X')$ stands for the conditional entropy of $\mu$ with respect to the sigma-algebra lifted from $X'$ against $\pi$.
\end{itemize}

From now on, we will mostly focus on the $G$-shift $(\Omega,\sigma)$ and the set $\M$ of all $\sigma$-\im s on $\Omega$.
For a number $0<c\le\log|\Lambda|$, let $\M_c=\{\mu\in\M: h(\mu)\ge c\}$. Since the entropy function is upper semicontinuous, $\M_c$ is a closed subset of $\M$, and thus it is a Baire space.

A measure $\mu\in\M$ is \emph{Bernoulli} if the measure-preserving system $(\Omega,\mu,\sigma)$ is isomorphic to an i.i.d.\ process. Equivalently, there exists a finite measurable partition $\P$ of the measure space $(\Omega,\mu)$ which is \emph{generating} (i.e., the partitions $g\P$ with $g\in G$ jointly generate the sigma-algebra in $(\Omega,\mu)$), and such that the partitions $g\P$ of $(\Omega,\mu)$ are independent as $g$ ranges over $G$.

The goal of this section is to prove the following fact:

\begin{thm}\label{T1} The set of Bernoulli measures with entropy $c$ is dense in $\M_c$.
\end{thm}

The proof is partitioned into four lemmas, a Construction, and a concluding~part.

\begin{lem}\label{l1}For every finite set $K\subset G$ and any $\gamma>0$,
there exists a tiling system $(\bar\Omega,\nu,\sigma)$ with $(K,\gamma)$-\inv\ shapes, of entropy smaller than $c$, and which is Bernoulli.
\end{lem}
\begin{proof}
Let $\lambda_p$ be the $(p,1-p)$-Bernoulli measure on $\Omega_2=\{0,1\}^G$. If $p$ is small enough, the entropy $h(\lambda_p)$ of the system $(\Omega_2,\lambda_p,\sigma)$ is smaller than $c$. Note that $\lambda_p$ is ergodic and free. By Theorem~\ref{conley}, the system $(\Omega_2,\lambda_p,\sigma)$ admits a tiling factor $(\bar\Omega,\nu,\sigma)$ with $(K,\gamma)$-\inv\ shapes. Since factor maps do not increase entropy and any factor of a Bernoulli system is Bernoulli\footnote{This appears in~\cite[Theorem~4, p.\,127]{8a} in greater  generality than countable amenable groups. For another treatment for countable amenable groups see~\cite{3a}.}, this tiling system satisfies the assertion.
\end{proof}

\noindent{\bf Construction.} Let $(\bar\Omega,\nu,\sigma)$ be a dynamical tiling with shapes $S_1,S_2,\dots,S_m$, and let $\Omega=\Lambda^G$ where $\Lambda$ is a finite alphabet. We will say that a $(\sigma\times\sigma)$-\im\ $\eta$ on $\bar\Omega\times\Omega$ is a \emph{lift} of $\nu$ if its marginal on the first coordinate $\bar\Omega$ is equal to~$\nu$. Clearly, for every lift $\eta$ of $\nu$, the projection $\pi_1:\bar\Omega\times\Omega\to\bar\Omega$ given by $\pi_1(\bar\omega,\omega)=\bar\omega$ is a tiling factor of $(\bar\Omega\times\Omega,\eta,\sigma\times\sigma)$, therefore it determines the sets $B_i\subset\bar\Omega\times\Omega$ for $1\le i\le m$ and the towers $S_iB_i$ as in Theorem~\ref{conley}. For each $1\le i\le m$, let $D_i$ denote the distribution induced by $\eta$ on $\Lambda^{S_i}$ conditioned on $B_i$, i.e.,
$$
D_i(\alpha) = \frac{\eta(B_i\cap[\alpha])}{\eta(B_i)}, 
$$
where $\alpha\in\Lambda^{S_i}$ and 
$[\alpha]=\{(\bar\omega,\omega)\in\bar\Omega\times\Omega:\omega|_{S_i}=\alpha\}$. Given $\alpha\in\Lambda^{S_i}$, the value of $D_i(\alpha)$ can be interpreted as the conditional probability that whenever $(\bar\omega,\omega)\in B_i$, in which case $(\bar\omega,\omega)|_{S_i}$ is a tile of $(\bar\omega,\omega)$, the projection of that tile on the second coordinate, i.e., $\omega|_{S_i}$, matches the block $\alpha$.
Next, let $\bar\eta$ be the lift of $\nu$ with the same conditional distributions $D_i$ as induced by $\eta$, but with the additional property that for any fixed $\bar\omega\in\bar\Omega$, the blocks $\alpha$ on the second coordinate of the tiles of $(\bar\omega,\omega)$ appear independently of each other. To be precise, let $\{\bar\eta_{\bar\omega}:\bar\omega\in\bar\Omega\}$ denote the disintegration of $\bar\eta$ with respect to $\nu$ (note that for $\nu$-almost every $\bar\omega\in\bar\Omega$, $\bar\eta_{\bar\omega}$ is a measure on $\Omega$). Given $\bar\omega\in\bar\Omega$ and $1\le i\le m$, let
$$
G_{\bar\omega,i}=\{g\in G: g\bar\omega\in B_i\}, 
$$
and let $G_{\bar\omega}=\bigcup_{i=1}^m G_{\bar\omega,i}$.
Then, for $\nu$-almost every $\bar\omega$ the sets 
$$
\{\omega:\omega|_{S_ig}=\alpha\},
$$
where $\alpha\in\Lambda^{S_i}$ and $i$ is such that $g\in G_{\bar\omega,i}$,
are $\bar\eta_{\bar\omega}$-independent for all $g\in G_{\bar\omega}$.

\begin{lem}\label{max}The measure $\bar\eta$ has maximal entropy among those lifts of $\nu$
which induce the same conditional distributions $D_i$.
\end{lem}

\begin{proof}
It suffices to show that $h(\bar\eta)\ge h(\eta)$, where $\eta$ and $\bar\eta$ are described in the Construction. Since $h(\eta)=h(\nu)+h(\eta|\bar\Omega)$, and
$$
h(\eta|\bar\Omega)=\limsup_{n\to\infty}\frac1{|F_n|}\int H_{\eta_{\bar\omega}}(\Lambda^{F_n})\,d\nu(\bar\omega),
$$
where $\{F_n\}_{n\in\N}$ is our fixed F\o lner \sq\ in $G$ (see \cite{13a}), it suffices to show that, for $\nu$-almost all $\bar\omega\in\bar\Omega$, any $\delta>0$ and large enough $n\in\N$, we have 
\begin{equation}\label{ineq}
\tfrac1{|F_n|}H_{\bar\eta_{\bar\omega}}(\Lambda^{F_n})\ge\tfrac1{|F_n|}H_{\eta_{\bar\omega}}(\Lambda^{F_n})-\delta.
\end{equation}
Fix some $\bar\omega\in\bar\Omega$ and a large $n\in\N$. We can decompose $F_n$ as the union of complete tiles of $\bar\omega$, i.e., sets of the form $S_ig$, where $g\in G_{\bar\omega,i}$ and $S_ig\subset F_n$, and some reminder $R$. Then, 
\begin{multline}\label{uno}
H_{\eta_{\bar\omega}}(\Lambda^{F_n})=H_{\eta_{\bar\omega}}\Bigl(\Lambda^R\vee\bigvee_{i=1}^m\ \bigvee_{\{g\in G_{\bar\omega,i}:\,S_ig\subset F_n\}}\Lambda^{S_ig}\Bigr)\le\\ 
H_{\eta_{\bar\omega}}(\Lambda^R)+\sum_{i=1}^m\ \sum_{\{g\in G_{\bar\omega,i}:\,S_ig\subset F_n\}}H(D_i).
\end{multline}
An identical formula holds for $\bar\eta$, except that, due to the conditional independence of the tiles, the inequality becomes equality:
\begin{equation}\label{dos}
H_{\bar\eta_{\bar\omega}}(\Lambda^{F_n})=H_{\bar\eta_{\bar\omega}}(\Lambda^R)+\sum_{i=1}^m\ \sum_{\{g\in G_{\bar\omega,i}:\,S_ig\subset F_n\}}H(D_i).
\end{equation}
Since $n$ is large, we can assume that $|R|$ is a small fraction of $|F_n|$, so that 
$H_{\eta_{\bar\omega}}(\Lambda^R)\le|R|\log|\Lambda|<|F_n|\delta$. The desired inequality~\eqref{ineq} now follows by juxtaposing~\eqref{uno} and~\eqref{dos}.
\end{proof}

\begin{lem}\label{l2}If the system $(\bar\Omega,\nu,\sigma)$ is Bernoulli, then the system $(\bar\Omega\times\Omega,\bar\eta,\sigma\times\sigma)$ constructed above is also Bernoulli.
\end{lem}
\begin{proof}Let us create a new huge (yet finite) alphabet, $\Delta=\prod_{i=1}^m\Lambda^{S_i}$. By definition, each symbol in $\Delta$ is an  $m$-tuple of blocks. On $\Delta$, consider the product distribution $D=\prod_{i=1}^mD_i$ and let $\tilde\lambda$ denote the Bernoulli measure $D^G$ on $\tilde\Omega=\Delta^G$. Clearly, the direct product $(\bar\Omega\times\tilde\Omega,\nu\times\tilde\lambda,\sigma\times\sigma)$ is a Bernoulli system. We will show that  $(\bar\Omega\times\Omega,\bar\eta,\sigma\times\sigma)$ is a factor of $(\bar\Omega\times\tilde\Omega,\nu\times\tilde\lambda,\sigma\times\sigma)$. Since any factor of a Bernoulli system is Bernoulli, this will end the proof. Let a factor map $\theta:\bar\Omega\times\tilde\Omega\to\bar\Omega\times\Omega$ be given by:
$$
\theta(\bar\omega,\tilde\omega) = (\bar\omega,\omega),
$$
where, $\tilde\omega\in\tilde\Omega$ and $\omega\in\Omega$ is defined by $\omega|_{S_ig}=\alpha\in\Lambda^{S_i}$, where $S_ig$ is a tile of $\bar\omega$ and $\alpha$ occurs as the $i$th coordinate in the symbol ($m$-tuple of blocks) $\tilde\omega_g$. Verification that $\theta$ is a factor map sending $\nu\times\tilde\lambda$ to $\bar\eta$ is straightforward.
\end{proof}

\begin{lem}\label{l3}
Let $\mu$ be any free ergodic shift-\im\ on $\Omega=\Lambda^G$ with entropy $h(\mu)>0$. For any $\eps>0$ and any $c\in(0,h(\mu)]$ there exists a measure-theoretic factor map $\phi:\Omega\to\Omega$ 
such that $\phi^*(\mu)$ is $3\eps$-close to $\mu$ and
$h(\phi^*(\mu))=c$.
\end{lem}
\begin{proof}We start by choosing a free factor $(X,\mu',\tau)$ of the system $(\Omega,\mu,\sigma)$ with entropy less than $c$.\footnote{The existence of such a factor is stated without proof on page 3 in \cite{8a} (it can be deduced from the paper). An explicit version can be found in \cite[Theorem 5.4]{3a}.}
By Theorem~\ref{conley}, the system $(X,\mu',\tau)$ splits as a family of disjoint towers with bases $B_1,B_2,\dots,B_m$ and $(K,\gamma)$-\inv\ shapes $S_1,S_2,\dots,S_m$ (the parameters $K$ and $\gamma$ will be specified in a moment). These towers can be lifted to become towers of the original system $(\Omega,\mu,\sigma)$. From now on, we will think of $B_1,B_2,\dots,B_m$ as subsets of $\Omega$. What we have gained is that the  entropy of the tiling system associated with these towers has entropy smaller than $c$.
Since $\mu$ is free, it is nonatomic, and thus, for each $1\le i\le m$, there exists a family of measurable subsets $B^{(t)}_i\subset B_i$, indexed by $t\in[0,1]$, such that
\begin{enumerate}
    \item $B^{(1)}_i= B_i$, $B^{(0)}_i\subset B_i=\varnothing$, 
    \item $t<t'\implies B^{(t)}_i\subset B^{(t')}_i$,
    \item $\mu(B^{(t)}_i)=t\mu(B_i)$.
\end{enumerate}
Since $\mu$ is ergodic, by an appropriate choice of $K$ and $\gamma$ we can guarantee that for each $1\le i\le m$, there exists a block $\alpha_i\in\Lambda^{S_i}$, whose empirical measure is $\eps$-close to $\mu$.\footnote{This follows by a standard argument from the mean ergodic theorem, see \cite[Theorem~3,~p.\,52]{8a}.} For each $t\in[0,1]$, we let $\phi_t:\Omega\to\Omega$ be the factor map defined $\mu$-almost everywhere by the following local rule: if $g^{-1}\omega\in B_i$ then we let
$$
\phi_t(\omega)|_{S_ig} = 
\begin{cases}
\alpha_i& \text{ if }g^{-1}\omega\in B^{(t)}_i,\\
\omega|_{S_ig}& \text{ otherwise}.
\end{cases}
$$

In steps (a)-(c) below we will show that $\phi_t^*(\mu)$ is $3\eps$-close~to~$\mu$.
\begin{enumerate}[(a)]
    \item Let $\omega\in\Omega$ be generic for $\mu$. The shapes are $(K,\gamma)$-\inv, therefore, those occurrences of a test block in $\omega$ which intersect more than one tile contribute less than $\delta$ to the total frequency of that block in $\omega$. Thus, the \emph{averaged empirical measure associated with the tiles} of $x$ (obtained by averaging\footnote{The averaging depends on the F\o lner \sq, but we have fixed one together with the group~$G$.} the empirical measures over all tiles of $x$), is $\eps$-close to~$\mu$.
    \item The mapping $\phi_t$ replaces some ``random'' tiles of $\omega$ by tiles that are $\eps$-close to $\mu$, so the averaged empirical measure associated with the tiles of $\phi_t(\omega)$ remains $2\eps$-close to $\mu$.\footnote{Think of the averaging as a probabilistic integration (over the tiles) of the function $d$ representing the distance to $\mu$. We know that the integral is at most $\eps$. The mapping $\phi_t$ corresponds to replacing the function $d$ on some subset by a function $c\le\eps$. Clearly, the integral of the new function does not exceed the sum of the integrals of $d$ and $c$, which is at most $2\eps$.}
    \item Clearly, $\phi_t(x)$ is generic for $\phi^*_t(\mu)$. Using the same argument as in (a), the averaged empirical measure associated with the tiles of $\phi_t(x)$ is $\eps$-close to $\phi_t^*(\mu)$.  
\end{enumerate}
Now, the $3\eps$-closeness of $\phi_t^*(\mu)$ and $\mu$ follows by composing (b) and (c).

Next, we claim that the entropy of $\phi_t^*(\mu)$ varies continuously from $h(\mu)$ to a number less than $c$, as $t$ ranges from $0$ to $1$. For each $t$, let $\P_t=\phi_t^{-1}(\Lambda)$, where $\Lambda$ is viewed as the partition of $\Omega$. For $t<t'$ the partitions $\P_t$ and $\P_{t'}$ differ on (a subset of) the set
$$
\bigcup_{i=1}^mS_i(B_i^{(t')}\setminus B_i^{(t)}),
$$
whose measure, as easily seen, is equal to $t'-t$. This implies that the mapping $t\mapsto\P_t$ is continuous with respect to the $d_1$-distance of partitions with equal cardinalities (see Definition~\ref{d1}). 
Since $\P_t$ is a generator for $\phi_t^*(\mu)$, we have $h(\phi_t^*(\mu))=h_\mu(\P_t)$. It is well known that the dynamical entropy of a partition, $h_\mu(\cdot)$, is continuous on the space of partitions of a fixed finite cardinality, equipped with the $d_1$-distance. We conclude that $h(\phi_t^*(\mu))$ is a continuous function of~$t$. Since $\P_0=\Lambda$, we have $h(\phi_0^*(\mu))=h(\mu)\ge c$. On the other hand, the partition $\P_1$ is measurable with respect to the tiling factor associated with the towers (which is factor of $(X,\mu',T)$) and thus $h(\phi_1^*(\mu))=h_\mu(\P_1)\le h(\mu')<c$. We can now invoke the intermediate value theorem, which implies that there exists a $t_0\in[0,1]$ such that $h(\phi_{t_0}^*(\mu))=c$. The map $\phi=\phi_{t_0}$ satisfies the assertion of the lemma.
\end{proof}

\begin{proof}[Proof of Theorem~\ref{T1}]
Choose $\mu\in\M_c$ and $\eps>0$. The proof will be completed once we construct a Bernoulli measure on $\Lambda^G$, which is $4\epsilon$-close to $\mu$ and has entropy equal to $c$.

Let $(\bar\Omega,\nu,\sigma)$ be a dynamical tiling with $(K,\gamma)$-\inv\ shapes $S_1,S_2,\dots,S_m$ and which is Bernoulli (see Lemma~\ref{l1}; this time we do not even require small entropy). For each shape $1\le i\le m$, let $D_{\mu,i}$ denote the distribution of $\mu$ on $\Lambda^{S_i}$. Let $(\bar\Omega\times\Omega,\bar\eta_\mu,\sigma\times\sigma)$ be the system created according to Construction. By Lemma~\ref{l2}, this system is Bernoulli, and by Lemma~\ref{max}, $\bar\eta_\mu$ has maximal entropy among all lifts of $\nu$ such that for every $1\le i\le m$ the distribution on $\Lambda^{S_i}$ conditioned on $B_i$ is equal to $D_{\mu,i}$. In particular, $\bar\eta_\mu$ satisfies 
$$
h(\bar\eta_\mu)\ge h(\nu\times\mu)=h(\nu)+h(\mu).
$$
Let $(\bar\omega,\omega)$ be generic for $\bar\eta$. The empirical measure on test blocks averaged over the tiles of $(\bar\omega,\omega)$ is the same as for any element generic for $\nu\times\mu$, because it depends only on $\nu$ and the distributions $D_{\mu,i}$. Since the shapes are $(K,\gamma)$-\inv, the projection $\bar\eta_\mu^{(2)}$ of $\bar\eta_\mu$ on the second coordinate $\Omega$ is $\eps$-close to $\mu$ (this follows by the same argument as used earlier, in step (a)). We also have 
$$
h(\bar\eta_\mu^{(2)})\ge h(\bar\eta_\mu)-h(\nu)\ge h(\mu)\ge c.
$$
Note that $\bar\eta_\mu^{(2)}$, being a factor of $\bar\eta_\mu$, is Bernoulli (hence free). By Lemma~\ref{l3} applied to $\bar\eta_\mu^{(2)}$, there exists a $\bar\eta_\mu^{(2)}$-almost everywhere defined factor map $\phi:\Omega\to\Omega$ such that $\phi^*(\bar\eta_\mu^{(2)})$ is $3\eps$-close to $\bar\eta_\mu^{(2)}$ (hence $4\eps$-close to $\mu$), and which has entropy equal to $c$. 
As a factor of a Bernoulli measure, $\phi^*(\bar\eta_\mu^{(2)})$ is Bernoulli, and hence it satisfies all properties specified at the beginning of the proof.
\end{proof}

\section{Bernoulli measures form a $G_\delta$-set}\label{S4}

By Theorem~\ref{T1}, the set of Bernoulli measures with entropy $c$ is dense in $\M_c$. In this section, we will prove that the same set is of type $G_\delta$. 

\begin{thm}\label{T2}
The Bernoulli measures of entropy $c$ form a $G_\delta$-set in $\M_c$. 
\end{thm}

The idea behind the proof of this theorem is to pass to a $\Z$-action (i.e., action generated by a single transformation) with the same orbits and to use techniques already established for $\Z$-actions.

Let $(X,\lambda,T)$ be an auxiliary symbolic Bernoulli $\Z$-action (for instance, the $(\frac12,\frac12)$-i.i.d.\ process of two symnols). There exists a Bernoulli $G$-action $\tau$ on $(X,\lambda)$ which has the same orbits as $(X,\lambda,T)$ (use \cite[Theorem 6]{7a} and Dye's Theorem \cite{4a}). Let $x\mapsto g_x$ be the (measurable) map that associates with $\lambda$-almost every $x\in X$ an element $g_x\in G$ such that $Tx = g_xx$. Then, for $\nu$-almost every $x\in X$, the map $\vartheta_x:\N\to G$, given by
$$
\vartheta_x(n) = 
\begin{cases}
g_{T^nx}g_{T^{n-1}x}\dots g_{Tx}g_x,& n\ge1,\\
e,& n=0,\\
g^{-1}_{T^{-n}x}g^{-1}_{T^{-n+1}x}\dots g^{-1}_{T^{-1}x},& n\le-1,\\
\end{cases}
$$

\smallskip\noindent
is a bijection and, for every $n\in\Z$, we have $T^nx=\vartheta_x(n)x$.
Let $\mu$ be a measure on $\Omega=\Lambda^G$ \inv\ under the $G$-shift $\sigma$, and consider the direct product of the $G$-shift $(\Omega,\mu,\sigma)$ with the Bernoulli $G$-action $(X,\lambda,\tau)$,  
$$
(X\times\Omega,\lambda\times\mu,\tau\times\sigma).
$$ 
On $X\times\Omega$ define the skew product transformation
$\bar T$ by 
\begin{equation}\label{Tbar}
\bar T(x,\omega)=(T(x),g_x(\omega))=g_x(x,\omega).
\end{equation}
Since $\vartheta_x$ is $\nu$-almost surely a bijection, the systems
$(X\times\Omega,\lambda\times\mu,\tau\times\sigma)$ and 
$(X\times\Omega,\lambda\times\mu,\bar T)$ are orbit equivalent (have the same orbits).

Recall that a skew product is \emph{relatively Bernoulli} with respect to the base if it is isomorphic to a direct product of the base with a Bernoulli system via an isomorphism which preserves the base factor. This  definition applies to actions of any groups that allow reasonable ergodic theory, in particular to countable amenable groups.

For brevity, from now we will denote $\lambda\times\mu$ by $\bar\mu$. 

\begin{lem}\label{L4}A measure
$\mu\in\M$ is Bernoulli under the action of the $G$-shift $\sigma$ if and only if $\bar\mu$ is Bernoulli relative to $\lambda$ under the $\Z$-action generated by $\bar T$. 
\end{lem}

\begin{proof}Obviously, if $\mu$ is Bernoulli under $\sigma$ then $\bar\mu$ is Bernoulli relative to $\lambda$ under $\tau\times\sigma$. Since orbit equivalence preserves the property of being relatively Bernoulli (see~\cite[Lemma~5.2]{1a}), the measure $\bar\mu$ is Bernoulli relative to $\lambda$ also under the $\Z$-action generated by $\bar T$.
Conversely, if $\bar\mu$ is Bernoulli relative to $\lambda$ under the $\Z$-action generated by $\bar T$ then $\bar\mu$ is Bernoulli relative to $\lambda$ under $\tau\times\sigma$. This means that the system $(X\times\Omega,\bar\mu,\tau\times\sigma)$ is isomorphic to a direct product $(X\times Y,\lambda\times\eta,\tau\times\pi)$, where $(Y,\eta,\pi)$ is some Bernoulli $G$-system\footnote{In general, the sole fact that $(X\times\Omega,\lambda\times\mu,\tau\times\sigma)$ is isomorphic to $(X\times Y,\lambda\times\eta,\tau\times\pi)$ does not imply that $(\Omega,\mu,\sigma)$ is isomorphic to $(Y,\eta,\pi)$. However, in this particular case it does, which follows from the rest of the proof and the Ornstein isomorphism theorem for amenable groups (see~\cite{8a}).}. This implies that the system
$(X\times\Omega,\bar\mu,\tau\times\sigma)$ is Bernoulli (because it is isomorphic to a direct product of two Bernoulli systems), and hence $(X,\mu,\sigma)$ is Bernoulli as a nontrivial factor of a Bernoulli system. This ends the proof. 
\end{proof}

Our goal is to show that Bernoulli measures of entropy $c$ form a $G_\delta$-set. By Lemma~\ref{L4} it suffices to show that the set of measures $\mu$ of entropy $c$ such that $\bar\mu$, viewed as a $\bar T$-\im, is Bernoulli relative to $\lambda$, form a $G_\delta$-set. This passage allows us to focus completely on the $\Z$-action given by the skew product transformation~$\bar T$. Our main tools are the notions of the relative very weak Bernoulli (VWB) property and conditional $\eps$-independence of partitions. We will apply a technique (invented by Dan Rudolph) which has been exploited e.g., in \cite{1a,5a,12a}. We will follow more closely \cite{12a}, as it contains the most detailed description of the argument.  

Since $(X,\lambda,T)$ is symbolic, it has a finite clopen generating partition $\P$. Also, $\Lambda$ viewed as a partition of $\Omega$, is clopen. 

\begin{lem}\label{gener}
The partition $\R=\P\times\Lambda$ is clopen and generating in $(X\times\Omega,\bar\mu,\bar T)$, for any $\mu\in\M$. 
\end{lem}
\begin{proof}It is obvious that $\R$ is clopen.
We need to show that $\R$ distinguishes the $\bar T$-orbits of $(x,\omega)\neq(x',\omega')$, except for pairs with $x\neq x'$ whose orbits are not distinguished by $\P$ (such pairs have $\bar\mu$-measure zero for any $\mu$). If the orbits of $x, x'$ are distinguished by $\P$, we are done. It remains to consider the case $x=x'$.
Then $\omega\neq\omega'$, while the bijections $\vartheta_x$ and $\vartheta_{x'}$ coincide. There exists $g\in G$ such that $\omega_g\neq\omega'_g$. Let $n=\vartheta_x^{-1}(g)$. Then the $\Omega$-components of $\bar T^n(x,\omega)$ and $\bar T^n(x,\omega')$ have different symbols at $e$, i.e., belong to different elements of the partition $\Lambda$.
\end{proof}

According to \cite{11a}, $\bar\mu$ is Bernoulli relative to $\lambda$ (under the action of $\bar T$) if and only if the partition $\R$ is VWB relative to $\P$. This criterion is applicable only if $\R$ and $\P$ are generating for the respective measures, this is why Lemma~\ref{gener} is crucial. The definition of the relative VWB property depends on the $\bar d$-distance between probability distributions on blocks: 

\begin{defn}\label{dbar}
For two blocks $\alpha,\beta\in\Lambda^n$,  $n\in\N$,
$\alpha=(\alpha_i)_{1\le i\le n},\ \beta=(\beta_i)_{1\le i\le n}$, the \emph{Hamming distance} is defined by
$$
\bar d_n(\alpha,\beta)=\tfrac1n|\{1\le i\le n:\alpha_i\neq\beta_i\}|.
$$
For two probability distributions $P,P'$ on $\Lambda^n$ we let
$$
\bar d_n(P,P') = \inf_\xi\int\bar d_n(\alpha,\beta)\, d\xi(\alpha,\beta),
$$
where $\xi$ ranges over all couplings $P\vee P'$, i.e., probability measures on $\Lambda^n\times\Lambda^n$ whose respective marginals are $P$ and $P'$.
\end{defn}

In the following definition, for $n_1,n_2\in\Z,\ n_1<n_2$, we will abbreviate the join of partitions $\bigvee_{i=n_1}^{n_2}\bar T^{-i}\R$ as $\R^{[n_1,n_2]}$ (a similar convention applies to the partition~$\P$).

\begin{defn}\label{VWB}
The partition $\R$ is very weakly Bernoulli relative to $\P$ (in the skew product system ($X\times\Omega,\bar\mu,\bar T)$) if, for every $\eps>0$, there exist natural numbers $N$ and $k_0$ such that for all $k\ge k_0$ there exists a family $\G\subset\R^{[-k,-1]}\vee\P^{[-k,k]}$ such that
\begin{align}
&\bar\mu(\bigcup\G)>1-\eps,\label{aa}\\
&\bar d_N(\R^{[0,N-1]}|A\cap B, \R^{[0,N-1]}|B)<\eps,\label{bb}\\ 
&\text{whenever }A\in\R^{[-k,-1]},\ B\in\P^{[-k,k]},\ A\cap B\in\G,\notag
\end{align}
where $\R^{[0,N-1]}|E$ stands for the distribution on $\R^{[0,N-1]}$ with respect to the normalized measure $\bar\mu$ restricted to a set $E\subset X\times\Omega$ of positive measure $\bar\mu$. 
\end{defn}

\begin{lem}\label{over}
Definition~\ref{VWB} admits an equivalent formulation in which condition~\eqref{bb} is replaced by
\begin{align}
&\bar d_N(\R^{[0,N-1]}|A\cap B, \R^{[0,N-1]}|A'\cap B)<2\eps,\label{bb1}\\ 
&\text{whenever }A,A'\in\R^{[-k,-1]},\ B\in\P^{[-k,k]},\ A\cap B,A'\cap B\in\G.\notag
\end{align}
\end{lem}

\begin{proof}
If~\eqref{bb} holds and $A,A'\in\R^{[-k,-1]},\ B\in\P^{[-k,k]},\ A\cap B,A'\cap B\in\G$, then \eqref{bb1} follows directly from the triangle inequality. Conversely, suppose~\eqref{bb1} holds for a set
$\G\subset\R^{[-k,-1]}\vee\P^{[-k,k]}$ satisfying $\bar\mu(\G)>1-\eps$. Consider the family $\B\subset\P^{[-k,k]}$ such that for each $B\in\B$ the union of the family $\A_B=\{A\in\R^{[-k,-1]}:A\cap B\in\G\}$ has measure at least $1-\sqrt\eps$. Then, as easily seen, $\bar\mu(\bigcup\B)>1-\sqrt\eps$. Let $\G_0 = \{A\cap B:B\in\B,A\in\A_B\}$.
Note that 
\begin{equation}\label{aax}\tag{7a}
\bar\mu(\bigcup\G_0)>1-2\sqrt\eps.
\end{equation}For any $B\in\B$ we have the following  equality for distributions 
\begin{multline*}
\R^{[0,N-1]}|B = \sum_{A'\in\R^{[-k,-1]}}\bar\mu(A')\cdot\R^{[0,N-1]}|A'\cap B=\\
\sum_{A'\in\R^{[-k,-1]}\setminus\A_B}\bar\mu(A')\cdot\R^{[0,N-1]}|A'\cap B +
\sum_{\A_B}\bar\mu(A')\cdot\R^{[0,N-1]}|A'\cap B.
\end{multline*}
By convexity of the metric $\bar d_N$, for any $A\in\A_B$ we get
\begin{equation}\label{aa1}\tag{8a}
\bar d_N(\R^{[0,N-1]}|A\cap B, \R^{[0,N-1]}|B)<\sqrt\eps + 2\eps,
\end{equation}
where $\sqrt\eps$ estimates the contribution of $\bar d_N(\R^{[0,N-1]}|A\cap B, \R^{[0,N-1]}|A'\cap B)$ (majorized by 1) over $A'\in\R^{[-k,-1]}\setminus\A_B$ and $2\eps$ estimates the contribution of the same terms over $A'\in\A_B$. Since in Definition~\ref{VWB} $\eps$ is arbitrary, we can now replace $\G$ by $\G_0$ and $\eps$ by $\max\{2\sqrt\eps,\sqrt\eps+2\eps\}$ to get an equivalent definition. 
\end{proof}

Since the involved partitions are clopen and both conditions~\eqref{aa} and~\eqref{bb1} are defined using sharp inequalities, they are fulfilled on open sets of measures. Alas, the configuration of quantifiers,
$$
\forall_{\eps}\ \exists_{N,k_0}\ \forall_{k\ge k_0}\ \exists_{\mathcal G,\mathcal Q} \ \eqref{aa}\wedge\eqref{bb1}
$$
does not represent a $G_\delta$-set. We need to eliminate the quatifier $\forall_{k\ge k_0}$. This is done using another (standard) distance between distributions and the notion of $\eps$-independence. 

\begin{defn}\label{dist}
Let $P,P'$ be probability distributions on $\Lambda^n$, $n\in\N$. We let
$$
\dist(P,P')=\sum_{\alpha\in\Lambda^n}|P(\alpha)-P'(\alpha)|.
$$
\end{defn}
It is well known that $\bar d_n(P,P')\le\dist(P,P')$.\footnote{In fact, we have $\bar d_n(P,P')\le\frac12\dist(P,P')$, here is why: For any coupling $P\vee P'$, the integral of $\bar d_n$ does not exceed the mass of the off-diagonal part of $\Lambda^n\times\Lambda^n$. There exists a \emph{maximal coupling} for which that mass equals exactly $\frac12\dist(P,P')$ (the existence of such a coupling is usually attributed to Dobrushin). This implies the  inequality in question.}  

\begin{defn}\label{epsind}
Let $\R_0,\R_1,\R_2$ be finite partitions of a probability space $(X,\mu)$. We say that \emph{$\R_0$ is conditionally $\eps$-independent of $\R_2$ given $\R_1$}, if there exists a family $\C$ of intersections $C_1\cap C_2$, $C_1\in\R_1$ and $C_2\in\R_2$, such $\mu(\bigcup\C)>1-\eps$ and for all $C_1\cap C_2\in\C$, we have
$$
\dist(\R_0|C_1\cap C_2,\R_0|C_1)|<\eps.
$$
\end{defn}

It is well known that for each $\eps>0$ there exists $\delta>0$ such that, \begin{gather}
\R_0\underset{^{\R_1}}\perp^{\!\!\delta}\R_2\implies H_\mu(\R_0|\R_1\vee\R_2)>H_\mu(\R_0|\R_1)-\eps,\text{ \ and}\notag\\
H_\mu(\R_0|\R_1\vee\R_2)>H_\mu(\R_0|\R_1)-\delta\implies\R_0\underset{^{\R_1}}\perp^{\!\!\eps}\R_2\label{hind}
\end{gather}
(we will use only~\eqref{hind}). 
Since $\delta$ depends on $\eps$ as well as on the cardinality of $\R$, we will write $\delta(\eps,|\R|)$.

\begin{lem}\label{L5}{\rm(cf.~\cite[Lemma~4]{12a} for the non-relative VWB property)} Consider the skew product system $(X\times\Omega,\bar\mu,\bar T)$ introduced prior to Lemma~\ref{L4}.
Assume that $h(\mu)=c$. The partition $\R=\P\times\Lambda$ is VWB relative to $\P$  if and only if for every $\eps>0$ there exist natural numbers $N$ and $k_0$ and a family $\G_1\subset\R^{[-k_0,-1]}\vee\P^{[-k_0,k_0]}$ such that
\begin{align}
&\bar\mu(\bigcup\G_1)>1-\eps,\label{cc}\\
&\bar d_N(\R^{[0,N-1]}|A_1\cap B_1, \R^{[0,N-1]}|A'_1\cap B_1) <\eps,\label{dd}\\ 
&\text{whenever }A_1,A_1'\in\R^{[-k_0,-1]},\ B_1\in\P^{[-k_0,k_0]},\ A_1\cap B_1, A'_1\cap B_1\in\G_1,\notag\\
&H_{\bar\mu}(\R^{[0,N-1]}|\R^{[-k_0,-1]}\vee\P^{[-k_0,k_0]})<Nc+\delta(\eps,|\R|^N).\label{ee}
\end{align}
\end{lem}

\begin{proof}
Throuoghout the proof, we will abbreviate $H_{\bar\mu}$ as $H$ and $h_{\bar\mu}$ as $h$. If $\R$ is VWB relative to $\P$ then, by~\eqref{aa} and~\eqref{bb1}, the conditions~\eqref{cc} and~\eqref{dd} are fulfilled for all sufficiently large $k_0$. On the other hand, as $k_0$ grows, the conditional entropies
\begin{equation}\label{seq}
H(\R^{[0,N-1]}|\R^{[-k_0,-1]}\vee\P^{[-k_0,k_0]})
\end{equation}
tend non-increasingly to 
$$
h(\bar T^N|X)=Nh(\bar T|X).
$$
By \cite[Theorem~2.6]{13a}, the conditional entropy is preserved under orbit equivalence, and hence $h(\bar T|X)=h(\tau\times\sigma|X)=h_\mu(\sigma)=c$, so the limit of \eqref{seq} equals $Nc$. Thus, for $k_0$ sufficiently large, \eqref{ee} holds as well.

Conversely, fix some $\eps> 0$ and suppose that the conditions~\eqref{cc}, ~\eqref{dd} and~\eqref{ee} are fulfilled for some $N$, $k_0$ and $\G_1$. Note that, by the aforementioned monotone convergence, for any $k\ge k_0$, we have
$$
Nc\le H(\R^{[0,N-1]}|\R^{[-k,-1]}\vee\P^{[-k,k]})\le Nc+\delta(\eps,|\R|^N). 
$$
In particular, we obtain that
$$
H(\R^{[0,N-1]}|\R^{[-k,-1]}\vee\P^{[-k,k]})\ge H(\R^{[0,N-1]}|\R^{[-k_0,-1]}\vee\P^{[-k_0,k_0]})-\delta(\eps,|\R|^N).
$$
Let us abbreviate 
\begin{gather*}
\R_0=\R^{[0,N-1]},\ \ \R_1=\R^{[-k_0,-1]}\vee\P^{[-k_0,k_0]},\text{\ \ and}\\ 
\R_2=\R^{[-k,-k_0-1]}\vee\P^{[-k,-k_0-1]\cup[k_0+1,k]}.
\end{gather*}
(in particular, $\G_1\subset\R_1$).
Note that $\R_1\vee\R_2=\R^{[-k,-1]}\vee\P^{[-k,k]}$.
In this notation, we have obtained
$$
H(\R_0|\R_1\vee\R_2)\ge H(\R_0|\R_1)-\delta(\eps,|\R_0|).
$$
In view of~\eqref{hind}, this means that given $\R_1$, $\R_0$ is conditionally $\eps$-independent of $\R_2$
(see Figure~\ref{fig1}).

\begin{figure}[h]
    \centering
 \includegraphics[width=\linewidth]{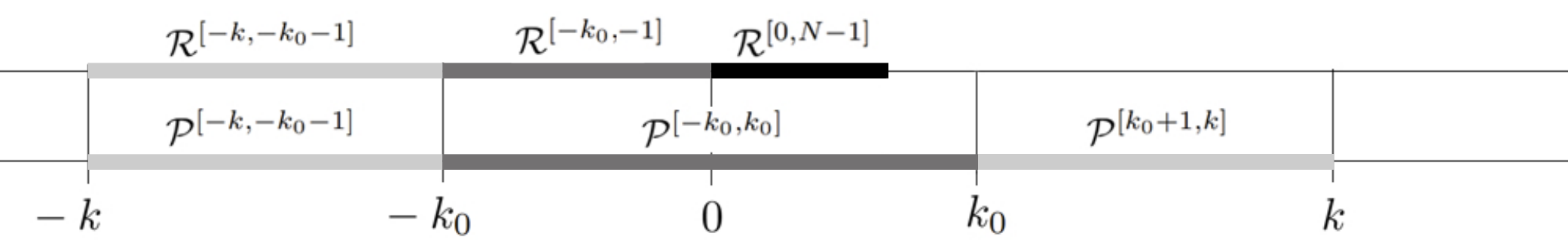}
    \caption{The top and bottom lines represent the evolution of $\R$ and $\P$, respectively. $\R_0$ is shaded black, $\R_1$ is shaded dark gray and $\R_2$ is shaded light gray. The dark gray area ``separates'' the black area from the light gray area, making them conditionally $\eps$-independent (this holds for any $\Z$-action and any two finite partitions $\R,\P$).}
    \label{fig1}
\end{figure}
By Definition~\ref{epsind}, this means that there exists a collection $\C$ of intersetions $C_1\cap C_2$, where $C_1\in\R_1$ and $C_2\in\R_2$, such that $\bar\mu(\bigcup\C)>1-\eps$, and for any $C_1\cap C_2\in\C$, we have
$$
\dist(\R_0|C_1\cap C_2,\R_0|C_1)<\eps,
$$
in particular, 
\begin{equation}\label{c1c2}
\bar d_N(\R_0|C_1\cap C_2,\R_0|C_1)<\eps.
\end{equation}
Let 
\begin{align*}
\G=\{A_1\cap B_1&\cap A_2\cap B_2:\\
&A_1\in\R^{[-k_0,-1]},\\ 
&B_1\in\P^{[-k_0,k_0]},\\ 
&A_2\in\R^{[-k,-k_0-1]},\\ 
&B_2\in\P^{[-k,-k_0-1]\cup[k_0+1,k]},\\
&\ \ \ \ \ A_1\cap B_1\in\G_1,\ \  A_1\cap B_1\cap A_2\cap B_2\in\C\}.
\end{align*}
Then 
$$
\G\subset\R^{[-k,-1]}\cap\P^{[-k,k]}, \ \
\mu(\bigcup\G)>1-2\eps,
$$
which is~\eqref{aa} (with the inessential replacement of $\eps$ by $2\eps$).

Each element of $\G$ has the form
$E=A\cap B$, where $A=A_1\cap A_2\in\R^{[-k,-1]}$, $B=B_1\cap B_2\in\P^{[-k,k]}$.
We can now prove~\eqref{bb1}, as follows. Let $A,A'\in\R^{[-k,-1]}$ and $B\in\P^{[-k,k]}$ be such that $A\cap B,A'\cap B\in\G$. Represent
$A=A_1\cap A_2,A'=A_1'\cap A_2'$ and $B=B_1\cap B_2$, as described above.
Denote $C_1=A_1\cap B_1$, $C'_1=A'_1\cap B_1$, $C_2=A_2\cap B_2$, $C'_2=A'_2\cap B_2$. Then $C_1,C_1'\in\R_1$, $C_2,C_2'\in\R_2$, $C_1\cap C_2\in\C$, $C_1'\cap C_2'\in\C$, and we proceed as follows:
\begin{multline*}
\bar d_N(\R^{[0,N-1]}|A\cap B, \R^{[0,N-1]}|A'\cap B)=\\  
\bar d_N(\R^0|A_1\cap A_2\cap B_1\cap B_2,\R^0|A'_1\cap A'_2\cap B_1\cap B_2)\le
\end{multline*}
\begin{multline*}
\bar d_N(\R^0|A_1\cap A_2\cap B_1\cap B_2,\R^0|A_1\cap B_1)+\\
\bar d_N(\R^0|A_1\cap B_1,\R^0|A'_1\cap B_1)+\\
\bar d_N(\R^0|A'_1\cap B_1,\R^0|A'_1\cap A'_2\cap B_1\cap B_2)=
\end{multline*}
\begin{multline*}
\bar d_N(\R^0|C_1\cap C_2,\R^0|C_1)+\\
\bar d_N(\R^0|A_1\cap B_1,\R^0|A'_1\cap B_1)+\\
\bar d_N(\R^0|C'_1,\R^0|C'_1\cap C'_2)< \eps+\eps+\eps,
\end{multline*}
by~\eqref{c1c2},~\eqref{dd}, and again~\eqref{c1c2}, respectively.  
We have proved~\eqref{bb1} with ($3\eps$ instead of $2\eps$), which ends the proof of the lemma. 
\end{proof}

We are now in a position to finalize the proof of Theorem~\ref{T2}.
\begin{proof}[Proof of Theorem~\ref{T2}] Let $(X,\lambda,T)$ be an auxiliary symbolic Bernoulli $\Z$-action with a clopen generating partition $\P$ and let $\tau$ be a Bernoulli $G$-action on $(X,\lambda)$ which has the same orbits as $(X,\lambda,T)$. 
Let $\bar T$ be the skew product transformation on $X\times\Omega$ defined in~\eqref{Tbar}. By Lemma~\ref{gener}, $\P\times\Lambda$ is a clopen generating partition in the system $(X\times\Omega,\bar\mu,\bar T)$, for any $\mu\in\M$.
Then the set of Bernoulli measures of entropy $c$ in $\M_c$ equals 
$$
\mathcal N_c\cap\bigcap_{n\in\N}\bigcup_{N\in\N}\bigcup_{k_0\in\N}\ \ \bigcup_{\G_1\in\R^{[-k_0,-1]}\vee\P^{[-k_0,k_0]}}\U(n,N,k_0,\G_1),
$$
where $\mathcal N_c=\{\mu\in\M: h(\mu)=c\}$ and $\U(n,N,k_0,\G_1)$ denotes the set of $\mu\in\M$ such that $\bar\mu=\lambda\times\mu$ satisfies the conditions~\eqref{cc},~\eqref{dd} and~\eqref{ee} of Lemma~\ref{L5} with $\G_1$ and $\eps=\frac1n$. Since the entropy function in a symbolic system is upper semicontinuous, $\mathcal N_c$ is of type $G_\delta$ in $\M$. The partitions $\R$ and $\P$ are clopen, and the map $\mu\mapsto\bar\mu$ is a homeomorphism onto its image. Thus, for each $n,N,k_0\in\N$ and $\G_1\in\R^{[-k_0,-1]}\vee\P^{[-k_0,k_0]}$, the set $\U(n,N,k_0,\G_1)$ is open in~$\M$.\footnote{The restriction to $\mathcal N_c$ is crucial. If, in condition~\eqref{ee}, we replace $c$ by $h(\mu)$, the respective set $\U(n,N,k_0,\G_1)$ ceases to be open.}
\end{proof}

\section{Relative Rudolph's theorem}\label{S5}

The proof of Theorem~\ref{Ru1} follows exactly the same scheme as that of Theorem~\ref{Ru}: we need to prove the obvious analogs of 
Theorems~\ref{T1} and~\ref{T2}. We will skip most of the details, instead we indicate the main necessary modifications. 

The analog of Theorem~\ref{T1} states that the set of relatively (with respect to $\mu_1$) Bernoulli measures with relative entropy $c$ is dense in $\M^{\mu_1}_c$. In the paragraph entitled Construction, we take, as before, the product $\bar\Omega\times\Omega$, but now $\eta$ is defined as a lift of $\nu\times\mu_1$. We then consider the distributions $D_{i,\beta}$, $1\le i\le m$, $\beta\in\Lambda_1^{S_i}$, on blocks $\alpha\in\Lambda^{S_i}$ conditioned on the set $B_i\times[\beta]$. The measure $\bar\eta$ is constructed by the same method using conditional independence. Next, we  prove an analog of Lemma~\ref{max}, that $\bar\eta$ has maximal entropy among the lifts of $\nu\times\mu_1$ with the same distributions $D_{i,\beta}$. This will allow us later to deduce that the conditional entropy $h(\bar\eta_{\mu}|\Omega_1)$ is at least $c$. In the proof of an analog of Lemma~\ref{l2}, that the measure $\bar\eta$ is relatively Bernoulli with respect to $\mu_1$, we use a relative version of the fact that any factor of a Bernoulli system is Bernoulli, that can be found as~\cite[Proposition~5, p. 198]{10a}, see also~\cite[Theorem~4.9, p. 1731]{3a}, and which we rewrite in our notation:
 
\begin{thm}\label{danpark}Let $\pi:(X,\mu,\tau)\to(X',\mu',\tau')$ and $\pi':(X',\mu',\tau')\to(X'',\mu'',\tau'')$ be factor maps between 
some measure-preserving $G$-actions. If $(X,\mu,\tau)$ is relatively Bernoulli with respect to $(X'',\mu'',\tau'',)$ then so is $(X',\mu',\tau')$.
\end{thm}

Let us state the analog of Lemma~\ref{l3} that we will need. The notation is as in Theorem~\ref{Ru1}. The proof is almost identical, with rather obvious modifications.

\begin{lem}\label{l3n}
Let $\mu$ be any free ergodic shift-\im\ on $\Omega$ with conditional entropy $h(\mu|\Omega_1)>0$. For any $\eps>0$ and any $c\in(0,h(\mu|\Omega_1)]$ there exists a measure-theoretic factor map $\phi:\Omega\to\Omega$, which preserves the fibers of $\pi$, and such that $\phi^*(\mu)$ is $3\eps$-close to $\mu$ and $h(\phi^*(\mu)|\Omega_1)=c$.
\end{lem}

The final step of the proof of the analog of Theorem~\ref{T1} is now straightforward.
\smallskip

Next, we pass to proving an analog of Theorem~\ref{T2}, i.e., the $G_\delta$ property of the set of relatively Bernoulli measures in $\M^{\mu_1}_c$. Since in the proof of Theorem~\ref{T2} we already deal with relative bernoullicity, the adaptation is quite easy. We start by proving the following analog of Lemma~\ref{L4}. The modifications are rather straightforward. In the last step we use again Theorem~\ref{danpark}. 

\begin{lem}\label{L4n}A measure $\mu\in\M$ is relatively Bernoulli with respect to $\mu_1$ under the action of the $G$-shift $\sigma$ if and only if $\bar\mu=\lambda\times\mu$ is Bernoulli relative to $\lambda\times\mu_1$ under the $\Z$-action generated by $\bar T$. 
\end{lem}

From here, the proof of the analog of Theorem~\ref{T2} is nearly identical, we only need to replace bernoullicity relative to $\lambda$ by 
bernoullicity relative to $\lambda\times\mu_1$.

\end{document}